\documentclass[a4paper]{article}

\usepackage[dvipdfm]{graphicx}
\usepackage{amsthm}
\usepackage{amsmath}
\usepackage{amssymb}
\usepackage{color, array}
\usepackage{url}
\usepackage{multirow}
\usepackage{algorithmic}
\usepackage{algorithm}
\usepackage{dcolumn}

\newtheorem{Definition}{Definition}[subsection]
\newtheorem{Theorem}[Definition]{Theorem}
\newtheorem{Lemma}[Definition]{Lemma}

\newtheorem{Example}[Definition]{Example}
\newtheorem{Remark}[Definition]{Remark}
\newtheorem{Proposition}[Definition]{Proposition}
\newcommand{\Ass}{{\mathop{\mathrm{Ass}}}}

\begin{document}
\title{Primary Decomposition of Symmetric Ideals}

\author{Yuki Ishihara\thanks{Tokyo University of Science, yishihara@rs.tus.ac.jp}}
\date{}

\maketitle             

\begin{abstract}
We propose an effective method for primary decomposition of symmetric ideals. Let $K[X]=K[x_1,\ldots,x_n]$ be the $n$-valuables polynomial ring over a field $K$ and $\mathfrak{S}_n$ the symmetric group of order $n$. We consider the canonical action of $\mathfrak{S}_n$ on $K[X]$ i.e. $\sigma(f(x_1,\ldots,x_n))=f(x_{\sigma(1)},\ldots,x_{\sigma(n)})$ for $\sigma\in \mathfrak{S}_n$. For an ideal $I$ of $K[X]$, $I$ is called {\em symmetric} if $\sigma(I)=I$ for any $\sigma\in \mathfrak{S}_n$. For a minimal primary decomposition $I=Q_1\cap \cdots \cap Q_r$ of a symmetric ideal $I$, $\sigma(I)=\sigma (Q_1)\cap \cdots \cap \sigma(Q_r)$ is a minimal primary decomposition of $I$ for any $\sigma\in \mathfrak{S}_n$. We utilize this property to compute a full primary decomposition of $I$ efficiently from partial primary components. We investigate the effectiveness of our algorithm by implementing it in the computer algebra system Risa/Asir.

\end{abstract}

\section{Introduction}
\label{sec:introduction}
Algebraic structures with symmetry are often treated in mathematics. For example, symmetric polynomials and ideals generated by them appear in invariant theory and Galois theory. For analyzing the structure of ideals, primary decompositions are well-known as one of the useful tools. In this paper, we prove good properties of primary decompositions of symmetric ideals and propose an effective algorithm for primary decompositions of those ideals. 

For a proper ideal $I$ of $K[x_1,\ldots,x_n]$, {\em a primary decomposition} of $I$ is a set of primary ideals $\{Q_1,\ldots,Q_k\}$ such that $I=Q_1\cap \cdots \cap Q_k$. Several algorithms for primary decomposition have been studied in \cite{Eisenbud1992,GIANNI1988149,KAWAZOE20111158,SHIMOYAMA1996247} and the algorithms are mainly based on Gr\"obner basis computations. However, Gr\"obner bases are incompatible with the symmetry in general, thus a specialized algorithm utilizing the symmetry can be effective for symmetric ideals.  

For a polynomial $f(x_1,\ldots,x_n)$ over a filed $K$, $f$ is called {\em symmetric} if $f(x_1,\ldots,x_n)=f(x_{\sigma(1)},\ldots,x_{\sigma(n)})$ for any permutation $\sigma$ over $\{1,2,\ldots,n\}$. For example, $f(x_1,x_2)=x_1^2+x_1x_2+x_2^2$ is a symmetric polynomial. The factorization of a symmetric polynomial has also a symmetric structure. For a factorization $f=g_1^{e_1}\cdots g_k^{e_k}$ of a symmetric polynomial $f$ with irreducible polynomials $g_i$, $f=\sigma(g_1)^{e_1}\cdots \sigma(g_k)^{e_k}$ is also a factorization of $f$ for any permutation $\sigma$ over $\{1,2,\ldots,n\}$. In other words, the symmetric group $\mathfrak{S}_n$ of degree $n$ acts on $\{g_1,\ldots,g_k\}$. Hence, we can define the equivalent classes $C_1,\ldots,C_l$ of $\{g_1,\ldots,g_k\}$ with respect to the equivalent relation $\sim$ where $g_i\sim g_j$ if and only if $\sigma(g_i)=g_j$ for some $\sigma\in \mathfrak{S}_n$. In each class $C_i$, we can compute other factors of $C_i$ from one factor of $C_i$ i.e. $C_i=\{\sigma (g_i)\mid \sigma\in \mathfrak{S}_n\}$ for $g_i\in C_i$. Therefore, we can reduce the computation for the factorization by the group action. 

We consider such symmetric structures in the ideals as well. For an ideal of $K[x_1,\ldots,,x_n]$, $I$ is called {\em symmetric} if $\sigma(I)=I$ for any permutation $\sigma$ over $\{1,2,\ldots,n\}$. For a minimal primary decomposition $\mathcal{Q}=\{Q_1,\ldots,Q_k\}$ of $I$, $\sigma (\mathcal{Q})=\{\sigma(Q_1),\ldots,\sigma(Q_k)\}$ is also a minimal primary decomposition of $I$. However, $\mathfrak{S}_n$ does not always act on $\mathcal{Q}=\{Q_1,\ldots,Q_k\}$ in general since primary decompositions of $I$ are not necessarily unique. In order to solve this problem, we introduce a notion of ``the quotient set of primary components of $I$'' which is uniquely determined from $I$. For the quotient set of primary components $\mathcal{Q}[I]=\{\mathcal{Q}_{P_1}[I],\ldots,\mathcal{Q}_{P_k}[I]\}$ of $I$, we can define the equivalent classes $C_1,\ldots,C_l$ of $\mathcal{Q}[I]$ with respect to the equivalent relation $\sim$ where $\mathcal{Q}_{P_i}[I]\sim \mathcal{Q}_{P_j}[I]$ if and only if $\sigma(\mathcal{Q}_{P_i}[I])=\mathcal{Q}_{P_j}[I]$ for some $\sigma\in \mathfrak{S}_n$. In each class $C_i$, we can compute other classes of primary components in $C_i$ from one class of a primary component in $C_i$. We say that $\{C_1,\ldots,C_l\}$ is {\em the orbit decomposition} of $I$. As in the case of symmetric polynomials, we can reduce the computation for the primary decomposition. For practical computations, we also consider other symmetric properties of the quotient set of primary components. We implemented the algorithm in the computer algebra system Risa/Asir. In a naive computer experiment, we examine its effectiveness in several examples. 

This paper is organized as follows. In Section~\ref{sec:mathematicalbasis}, we recall some fundamental notions and definitions for symmetric ideals and primary decompositions. In Section~\ref{sec:PrimaryDecompositionofSymmetricIdeals}, we introduce the quotient set of primary components of a symmetric ideal for effective primary decompositions. In Section~\ref{sec:ImprovementsforSymmetricPrimaryDecomposition}, we provide some improvements for symmetric primary decompositions toward practical algorithms. In Section~\ref{sec:Experiment}, we implement our algorithm in the computer algebra system Risa/Asir \cite{risaasir} and examine the effectiveness of our algorithm in several examples. In Section~\ref{sec:ConclusionsandFutureWorks}, we summarize the conclusion and discuss future works. 

\section{Mathematical Basis}
\label{sec:mathematicalbasis}

We let $X=\{x_1,\ldots,x_n\}$ be a set of $n$-valuables and $K[X]$ the polynomial ring over a field $K$. Also, let $\mathfrak{S}_n$ be the symmetric group of degree $n$ and $\phi:\mathfrak{S}_n\times K[X]\to K[X]$ the canonical action of $\mathfrak{S}_n$ on $K[X]$ such that $\sigma (f(x_1,\ldots,x_n))=f(x_{\sigma(1)},\ldots,\sigma_{\sigma(n)})$ for $\sigma\in \mathfrak{S}_n$ and $f\in K[X]$. Here, for distinct $i_1,\ldots,i_k\in \{1,\ldots,n\}$, $(i_1 \, i_2\, i_3 \,\cdots\, i_k)$ is a permutation such that $i_1\mapsto i_2$, $i_2\mapsto i_3$, $\ldots$ $,i_{k-1}\mapsto i_k$, $i_k\mapsto i_1$. For polynomials $f_1,\ldots,f_l$, we denote by $\langle f_1,\ldots,f_l\rangle$ the ideal generated by them. For an ideal $I$, we call $\{f\in K[X]\mid f^m \in I \text{ for a positive integer $m$}\}$ the radical of $I$ and denote it by $\sqrt{I}$. Also, we call $I:J=\{f\mid fJ\subset I\}$ and $I\colon J^\infty=\{f\mid f^mJ\subset I, \text{ for a positive integer $m$}\}$ {\em the quotient} and {\em the saturation} of ideals $I$ and $J$ respectively.

\subsection{Symmetric Ideal}

First, we introduce a notion of {\em symmetric ideal}. 

\begin{Definition}
    For an ideal $I$ of $K[X]$, $I$ is called {\em symmetric} if $\sigma(I)=I$ for any $\sigma\in \mathfrak{S}_n$, where $\sigma(I)=\{\sigma(f) \mid f \in I \}$. 
\end{Definition}

\begin{Example}
    $I=\langle x_1^2-x_2^2,x_1x_2\rangle\subset K[x_1,x_2]$ is a symmetric ideal. 
\end{Example}

\begin{Remark}\label{rem:symcon}
    Since $\sigma$ is invertible, $\sigma(I)\subset I$ implies $I\subset \sigma^{-1}(I)$. Thus, the condition $\sigma(I)=I$ can be replaced by $\sigma(I)\subset I$. 
\end{Remark}

We generalize the term ``symmetric ideal'' and consider an action by a subgroup of $\mathfrak{S}_n$ on $I$. 

\begin{Definition}
    For an ideal $I$ of $K[X]$ and a subgroup $G$ of $\mathfrak{S}_n$, $I$ is called {\em $G$-invariant} if $\sigma(I)=I$ for any $\sigma\in G$.
\end{Definition}

\begin{Example}
    For $G=\{(1),(1 \, 2\, 3),(1 \, 3\, 2)\}\subset \mathfrak{S}_3$ and $I=\langle x_1x_2,x_2x_3,x_3x_1\rangle\subset K[x_1,x_2,x_3]$, $I$ is a $G$-invariant ideal. 
\end{Example}

\begin{Remark}
   Here, symmetric ideals are not necessarily generated from symmetric polynomials. For example, $I=\langle x-y \rangle$ is generated by non-symmetric polynomial $x-y$ but it is a symmetric ideal.
\end{Remark}

Next, we recall a definition of primary decomposition and prime divisors of an ideal as follows. 

\begin{Definition}
    Let $I$ be a proper ideal of $K[X]$. A set $\{Q_1,\ldots,Q_k\}$ of primary ideals is called {\em a primary decomposition of $I$} if 
    \[
    I=Q_1\cap \cdots \cap Q_k.
    \]
    A primary decomposition $\{Q_1,\ldots,Q_k\}$ of $I$ is called {\em minimal} or {\em irredundant} if $\sqrt{Q_i}\neq \sqrt{Q_j}$ for any pair $(i,j)$ with $i \neq j$ and $Q_i\not \supset \bigcap_{j\neq i} Q_j$ for any $i$. Each $Q_i$ is called {\em a ($\sqrt{Q_i}$-)primary component} of $I$ and $\sqrt{Q_i}$ is called {\em a prime divisor} or {\em an associated prime} of $I$. 
\end{Definition}

We denote by $\Ass (I)$ the set of prime divisors of $I$ i.e. $\Ass (I)=\{\sqrt{Q_1},\ldots,\sqrt{Q_k}\}$ for a minimal primary decomposition $\{Q_1,\ldots,Q_k\}$ of $I$. In the set of prime divisors, prime divisors which are minimal under set inclusion are called {\em isolated}, others are called {\em embedded} respectively. 

\begin{Remark}
    A minimal primary decomposition of $I$ is not always unique in general. For instance, $I=\langle x_1^2,x_1x_2\rangle$ has infinitely many primary decompositions of type $\{\langle x_1\rangle,\langle x_1^2,x_1x_2,x_2^m\rangle\}$ for any positive integer $m$. In Section~\ref{sec:PrimaryDecompositionofSymmetricIdeals}, we define the set of all $P$-primary components with respect to a prime divisor $P$ to avoid the non-uniqueness of primary decompositions. 
\end{Remark}

\subsection{Criteria for Symmetric Ideal}

In order to check whether a given ideal is symmetric or not, one can utilize a Gr\"obner basis of the ideal. For permutations $\sigma_1,\ldots,\sigma_l$ of $\mathfrak{S}_n$, we denote by $\langle\langle \sigma_1,\ldots,\sigma_l\rangle\rangle$ the subgroup generated by them. 

\begin{Lemma}
    Let $I=\langle f_1,\ldots,f_k\rangle$ be an ideal of $K[X]$ and $G$ a subgroup of $\mathfrak{S}_n$. Then, $I$ is $G$-invariant if and only if $\sigma(f_i)\in I$ for any $i\in \{1,\ldots,k\}$ and $\sigma\in G$. In particular, one can check whether a given ideal $I$ is $G$-invariant or not.   
\end{Lemma}

\begin{proof}
    If $I$ is $G$-invariant, then it is obvious that $\sigma( f_i)\in I$ for any $i$ and $\sigma\in G$. Suppose that $\sigma(f_i)\in I$ for any $i$ and $\sigma\in G$. Then, for $f\in I$, there exist $h_1,\ldots,h_k\in K[X]$ such that $f=h_1f_1+\cdots +h_kf_k$. Since $\sigma (f)=\sigma(h_1)\sigma(f_1)+\cdots +\sigma(h_k)\sigma(f_k)\in \langle \sigma(f_1),\ldots,\sigma(f_k)\rangle\subset I$, we obtain that $\sigma(I)\subset I$ and $I$ is $G$-invariant by Remark~\ref{rem:symcon}. In order to check whether $f\in I$ or $f\not \in I$, one can use a Gr\"obner basis of $I$ with respect to a monomial ordering on $K[X]$. 
\end{proof}

One can check if $I$ is $G$-invariant or not more easily by using a generating set of $G$ as follows. 

\begin{Lemma}
    Let $I=\langle f_1,\ldots,f_k\rangle$ be an ideal of $K[X]$ and $G$ be a subgroup of $\mathfrak{S}_n$. If $G$ is generated by $\sigma_1,\ldots,\sigma_l$ then $I$ is $G$-invariant if and only if $\sigma_j(f_i)\in I$ for any $i\in \{1,\ldots,k\}$ and $j\in \{1,\ldots,l\}$. 
\end{Lemma}
\begin{proof}
    It is obvious that $\sigma_j(f_i)\in I$ if $I$ is $G$-invariant. Suppose that $\sigma_j(f_i)\in I$ for any $i\in \{1,\ldots,k\}$ and $j\in \{1,\ldots,l\}$. For $\sigma\in G$, there exist $\sigma_{i_1},\ldots,\sigma_{i_m}$ $(1\le i_1,\ldots,i_m\le l)$ such that $\sigma=\sigma_{i_1}\cdots\sigma_{i_m}$. Thus, for any $\sigma\in G$ and $f\in I$, it follows that $\sigma(f)=\sigma_{i_1}(\cdots(\sigma_{i_m}(f)))\in I$ and $\sigma(I)\subset I$. 
\end{proof}

\begin{Example}
    Since $\mathfrak{S}_n$ is generated by $(1 \, 2)$ and $(1 \, 2\, 3\, \cdots n)$, $I$ is symmetric if and only if $(1 \, 2) (I)=I$ and $(1 \, 2\, 3\, \cdots n) (I)=I$. 
\end{Example}

Finally, we obtain a decision algorithm for $G$-invariant ideals as Algorithm~\ref{alg:1}. 

\begin{algorithm}                      
\caption{\textsc{IsInvariantIdeal}}                                 
\begin{algorithmic}        \label{alg:1}
\REQUIRE $\{f_1,\ldots,f_k\}$: a set of polynomials. $\sigma_1,\ldots,\sigma_l$: permutations of $\mathfrak{S}_n$
\ENSURE 1 if $\langle f_1,\ldots,f_k\rangle$ is $\langle\langle \sigma_1,\ldots,\sigma_l\rangle\rangle$-invariant; 0 otherwise
\FOR{$i=1$ to $k$}
    \FOR{$j=1$ to $l$}
        \IF{$\sigma_j(f_i)\not \in I$}
            \RETURN 0
        \ENDIF
    \ENDFOR
\ENDFOR
\RETURN 1
\end{algorithmic}
\end{algorithm}

\section{Primary Decomposition of Symmetric Ideals}
\label{sec:PrimaryDecompositionofSymmetricIdeals}

In this section, we reveal a symmetric structure in primary divisors of a symmetric ideal. We remark that $\sigma\in \mathfrak{S}_n$ is an automorphism of $K[X]$, and thus $P$ is a prime ideal if and only if $\sigma(P)$ is a prime ideal. Also, $Q$ is a primary ideal if and only if $\sigma(Q)$ is a primary ideal. Similarly, other algebraic property of an ideal $I$ holds for $\sigma (I)$. In addition, $\sigma$ is commutative with many ideal operations, for example, $\sigma(I\cap J)=\sigma(I)\cap \sigma (J)$ and $\sigma(I+J)=\sigma(I)+\sigma(J)$. 

\subsection{Group Action on Primary Components}

First, we show that primary decomposition is commutative with a group action. 

\begin{Lemma}\label{lem:minsym}
    Let $I$ be a $G$-invariant ideal and $\sigma\in G$. For a primary decomposition $\mathcal{Q}=\{Q_1,\ldots,Q_k\}$ of $I$, $\sigma(\mathcal{Q})=\{\sigma(Q_1),\ldots,\sigma(Q_k)\}$ is also a primary decomposition of $I$. If $\mathcal{Q}$ is minimal, then $\sigma(\mathcal{Q})$ is also minimal. 
\end{Lemma}

\begin{proof}
    Since $\sigma$ is commutative with the ideal intersection, 
    \[
    I=\sigma(I)=\sigma(Q_1)\cap \cdots\cap\sigma(Q_k).
    \]
    Here, each $\sigma(Q_i)$ is primary and thus $\sigma(\mathcal{Q})$ is a primary decomposition of $I$. If $\mathcal{Q}$ is minimal then $\sqrt{\sigma(Q_i)}=\sigma(\sqrt{Q_i})\neq \sigma(\sqrt{Q_j})=\sqrt{\sigma(Q_j)}$ for any pair $(i,j)$ with $i \neq j$ and $\sigma(Q_i)\not \supset \sigma\left(\bigcap_{j\neq i} Q_j\right)=\bigcap_{j\neq i} \sigma(Q_j)$ for any $i$. Thus, $\sigma(\mathcal{Q})$ is also minimal. 
\end{proof}

Next, we prove that the set of prime divisors of a symmetric ideal has also a symmetric structure.   

\begin{Proposition} \label{prop:assact}
    Let $I$ be a $G$-invariant ideal and $\Ass (I)=\{P_1,\ldots,P_k\}$. Then, $G$ acts on $\Ass (I)$ by $\sigma(P_i)$ for $\sigma\in G$ and $P_i\in \Ass (I)$. 
\end{Proposition}

\begin{proof}
    It is enough to show that $\sigma(P_i)\in \Ass(I)$ for any $\sigma\in G$. For a minimal primary decomposition $\{Q_1,\ldots,Q_k\}$ of $I$, $\{\sigma(Q_1),\ldots,\sigma(Q_k)\}$ is also a minimal primary decomposition of $\sigma(I)$ by Lemma~\ref{lem:minsym} and thus $\sigma(P_i)\in \Ass (\sigma(I))$. Since $\sigma(I)=I$ for any $\sigma\in G$, we obtain $\sigma(P_i)\in \Ass (\sigma(I))=\Ass (I)$. 
\end{proof}

\begin{Example}
    Let $I=\langle x_1x_2,x_2x_3,x_3x_1\rangle\subset \mathbb{Q}[x_1,x_2,x_3]$ and $G=\langle\langle (1\;2\;3)\rangle\rangle$. Then, $I$ is $G$-invariant and  $I=\langle x_1,x_2\rangle\cap \langle x_2,x_3\rangle\cap \langle x_3,x_1\rangle$ i.e. $\Ass (I)=\{\langle x_1,x_2\rangle,\langle x_2,x_3\rangle,\langle x_3,x_1\rangle\}=\{P_1,P_2,P_3\}$. Here, for $\sigma=(1\;2\;3)$, we obtain that $\sigma(P_1)=P_2$, $\sigma(P_2)=P_3$ and $\sigma(P_3)=P_1$. 
\end{Example}

\begin{Remark} \label{rem:pri-nonuni}
    In general, $\mathcal{Q}=\mathcal{\sigma (Q)}$ is not always true. For example, $I=\langle x_1+x_2 \rangle\cap \langle x_2^3+x_2+1,(x_1+x_2)^2 \rangle\cap \langle x_2^3+x_2-1,(x_1+x_2)^2\rangle$ is a symmetric ideal and has a minimal primary decomposition $\mathcal{Q}=\{\langle x_1+x_2 \rangle,\langle x_2^3+x_2+1,(x_1+x_2)^2 \rangle, \langle x_2^3+x_2-1,(x_1+x_2)^2\rangle\}$. However, $\mathcal{Q}\neq \sigma(\mathcal{Q})$ for $\sigma = (1\,2)$. Indeed, $\sigma(\mathcal{Q})=\{\langle x_2+x_1 \rangle,\langle x_1^3+x_1+1,(x_2+x_1)^2 \rangle, \langle x_1^3+x_1-1,(x_2+x_1)^2\rangle\}$ and $\langle x_1^3+x_1+1,(x_2+x_1)^2 \rangle\in \sigma(\mathcal{Q})$ is not in $\mathcal{Q}$. 
\end{Remark}

In Remark~\ref{rem:pri-nonuni}, we see that $G$ does not always act on a primary decomposition of a $G$-invariant ideal $I$. Thus, we extend the notion of the primary decomposition as follows. 

\begin{Definition}
    Let $P$ be a prime divisor of $I$. We call the set of all $P$-primary components of $I$ {\em the class of a ($P$-)primary component} of $I$ and denote it by $\mathcal{Q}_P[I]$. We call the set of all classes of primary components $\{\mathcal{Q}_P[I]\mid P\in \Ass(I)\}$ {\em the quotient set of primary components} of $I$ and denote it by $\mathcal{Q} [I]$. 
\end{Definition}

\begin{Example}
    Let $I=\langle x_1^2,x_1x_2\rangle\subset \mathbb{Q}[x_1,x_2]$. Then, $\Ass (I)=\{\langle x_1\rangle,\langle x_1,x_2\rangle\}=\{P_1,P_2\}$. Here, $\mathcal{Q}_{P_1}[I]=\{\langle x_1\rangle\}$ and $\mathcal{Q}_{P_2}[I]$ contains $\langle x_1,x_1x_2,x_2^m\rangle$. In general, it is very difficult to express the elements of $\mathcal{Q}_{P}[I]$ explicitly. However, it is enough to know just one element of $\mathcal{Q}_{P_i}[I]$ for each $i$ for computing a primary decomposition of $I$ by Proposition~\ref{prop:29}. 
\end{Example}

Regardless of the choice of each primary component $Q_P$ of $\mathcal{Q}_{P}[I]$, the set $\{Q_P\mid P\in \Ass (I)\}$ is a minimal primary decomposition. The following proposition is proved directly from Proposition~2.9 in \cite{Ishihara2022}. 

\begin{Proposition}[\cite{Ishihara2022}, Proposition~2.9] \label{prop:29}
    Let $\mathcal{Q}[I]=\{\mathcal{Q}_{P_1}[I],\ldots,\mathcal{Q}_{P_k}[I]\}$ and $Q_i$ an element of $\mathcal{Q}_{P_i}[I]$ for each $i\in \{1,\ldots,k\}$. Then, $\{Q_1,\ldots,Q_k\}$ is a minimal primary decomposition of $I$. 
\end{Proposition}

Finally we obtain the action on the quotient set of primary components of $I$. 

\begin{Theorem}
    Let $I$ be a $G$-invariant ideal. Then, $G$ acts on $\mathcal{Q}[I]$ by $\sigma(\mathcal{Q}_P[I])$ for $\sigma\in G$ and $\mathcal{Q}_P[I]\in \mathcal{Q}[I]$. 
\end{Theorem}

\begin{proof}
    Let $P$ be a prime divisor and $Q$ a $P$-primary component of $I$. Then, $\sigma (P)$ is a prime divisor and $\sigma (Q)$ is a $\sigma(P)$-primary component of $I$ by Proposition~\ref{prop:assact}. Thus, $\sigma(\mathcal{Q}_P[I])=\mathcal{Q}_{\sigma(P)}[I]\in \mathcal{Q}[I]$. 
\end{proof}

\subsection{Algorithm for Primary Decomposition of Symmetric Ideals}
Here, we devise an effective algorithm specialized to symmetric ideals. For a $G$-invariant ideal, we only need to compute $l$ primary components where $l$ is the number of the {\em orbit decomposition} of $I$ with respect to $G$ in Definition~\ref{def:orbit}. In particular, if $G$ acts on $\Ass (I)$ transitively, then we can compute a minimal primary decomposition from just one primary component of $I$ (see Example \ref{ex:cyc}). 

\begin{Definition}\label{def:orbit}
    Let $\sim$ be an equivalent relationship between $\mathcal{Q}[I]$ defined by $\mathcal{Q}_{P_i}[I]\sim \mathcal{Q}_{P_j}[I]$ if and only if $\sigma(\mathcal{Q}_{P_i}[I])=\mathcal{Q}_{P_j}[I]$ for some $\sigma\in G$. We call the set $\{C_1,\ldots,C_l\}$ of all equivalent classes of $\mathcal{Q}[I]$ with respect to $\sim$ {\em the orbit decomposition of $I$ with respect to $G$}. 
\end{Definition}

\begin{Example}
    Let $I=\langle (x_1+1)(x_2+1)(x_1+x_2)\rangle\subset \mathbb{Q}[x_1,x_2]$. Then, $I$ is a symmetric ideal and $\Ass (I)=\{\langle x_1+1\rangle,\langle x_2+1\rangle,\langle x_1+x_2\rangle\}=\{P_1,P_2,P_3\}$. Then, for $C_1=\{\mathcal{Q}_{P_1}[I],\mathcal{Q}_{P_2}[I]\}$ and $C_2=\{\mathcal{Q}_{P_3}[I]\}$, it follows that $\{C_1,C_2\}$ is the orbit decomposition of $I$ with respect to $\mathfrak{S}_2$. 
\end{Example}

In each orbit $C_i$, one can compute other classes of primary components from one class of a primary component $\mathcal{Q}_{P_i}[I]$ in $C_i$ since $C_i=\{\sigma(\mathcal{Q}_{P_i}[I])\mid \sigma\in G\}$. Hence, we can compute a minimal primary decomposition from $l$-primary components $Q_1,\ldots,Q_l$, where each $Q_i$ is in $\mathcal{Q}_{P_i}[I]$ respectively. Here, Algorithm~\ref{alg:2} is an outline of an algorithm for primary decompositions of symmetric ideals. 

\begin{algorithm}[H]                      
\caption{\textsc{SymmetricPrimaryDecomposition}}                                 
\begin{algorithmic}        \label{alg:2} 
\REQUIRE $I$: a $G$-invariant ideal of $K[X]$. $G$: a subgroup of $\mathfrak{S}_n$
\ENSURE a minimal primary decomposition of $I$
\STATE $\mathcal{Q}=\{\}$
\WHILE{$\mathcal{Q}$ is not a minimal primary decomposition of $I$}
    \STATE Compute a primary component $Q$ of $I$, whose radical $\sqrt{Q}$ is not in $\{\sqrt{Q^\prime}\mid Q^\prime\in \mathcal{Q}\}$
    \STATE $\mathcal{Q}\gets \mathcal{Q}\cup \{\sigma (Q)\mid \sigma\in G\}$
\ENDWHILE
\RETURN $\mathcal{Q}$
\end{algorithmic}
\end{algorithm}

\begin{Example} \label{ex:cyc}
    	Let $I=cyclic(3)=\langle x_1x_2x_3-1,x_1x_2+x_2x_3+x_3x_1,x_1+x_2+x_3 \rangle\subset \mathbb{Q}[x_1,x_2,x_3]$ (see \cite{Backelin1991} for the definition of $cyclic(n)$). Then, $I$ is a symmetric ideal. By computing a Gr\"obner basis of $I$, it follows that 
		\[
		I\cap \mathbb{Q}[x_3]=\langle x_3^3-1 \rangle=\langle x_3-1 \rangle\cap \langle x_3^2+x_3+1 \rangle. 
		\]		
	    Then, $I=(I+\langle x_3-1 \rangle)\cap (I+\langle x_3^2+x_3+1 \rangle)$ and we obtain a primary component
		\[
		Q_1=(I+\langle x_3-1 \rangle)=\langle  x_2^2+x_2+1,x_1+x_2+1,x_3-1\rangle
		\]
       By acting $G$ on $Q_1$, we obtain other primary components
		\begin{align*}
			Q_2=(1\, 2\, 3) (Q_1)=\langle x_3^2+x_3+1,x_2+x_3+1,x_1-1 \rangle		\\
			Q_3=(1\, 3\, 2) (Q_1)=\langle x_1^2+x_1+1,x_3+x_1+1,x_2-1 \rangle
		\end{align*}
        and a minimal primary decomposition
		\[
		I=Q_1\cap Q_2\cap Q_3. 
		\]
\end{Example}

\section{Improvements for Symmetric Primary Decomposition}
\label{sec:ImprovementsforSymmetricPrimaryDecomposition}

In this section, we devise practical techniques and propose several improvements for symmetric primary decomposition. Here, we modify Shimoyama-Yokoyama Algorithm (SY-Algorithm) \cite{SHIMOYAMA1996247}, one of the effective algorithms for primary decomposition, by specializing in symmetric ideals. After a brief review of the SY-algorithm, we introduce a symmetric ideal version of the SY-algorithm.

\subsection{Outline of Shimoyama-Yokoyama Algorithm}

First, we recall an outline of SY-algorithm. For an ideal $I$, SY-algorithm uses the prime decomposition of $\sqrt{I}$ to compute {\em the pseudo-primary decomposition} of $I$. In more detail, it utilizes {\em separators} to compute {\em pseudo-primary components}. An ideal $I$ is called a {\em pseudo-primary} ideal if $\sqrt{I}$ is a prime ideal (see Definition~2.3 in \cite{SHIMOYAMA1996247}). 

\begin{Definition}[\cite{SHIMOYAMA1996247}, Definition~2.5 and Definition~2.8]
		Let $I$ be an ideal, which is not a pseudo-primary ideal, $P_1,\ldots,P_k$ all isolated prime divisors of $I$, and $S_1,\ldots,S_r$ are finite subsets in $K[X]$.  Each $S_i$ is called a {\em separator} of $I$ with respect to $P_i$ if they satisfy the following conditions;
		\[
		S_i\cap P_i=\emptyset,  \text{ and } S_i\cap P_j\neq \emptyset \text{ for } i\neq j.
		\] 
		A set of separators $\{S_1,\ldots,S_r\}$ is called a {\em system of separators} of $I$. For a separator $S_i$ of $I$ with respect to $P_i$ and $s_i=\prod_{s\in S_i}s$, we say that $\overline{Q}_i=I:s_i^\infty$ is a {\em $P$-pseudo-primary component} of $I$. Also, there exists an ideal $I^\prime$ of $K[X]$ such that 
  \[
  I=\overline{Q}_1\cap \cdots \cap \overline{Q}_r \cap I^\prime
  \]
  and this decomposition is called {\em a pseudo-primary decomposition} of $I$. Here, $I^\prime$ is called {\em the remaining component} of the pseudo-primary decomposition of $I$. 
\end{Definition}

For a pseudo-primary component $\overline{Q}_i$ of $I$, we can compute the isolated primary component $Q_i$ of $\overline{Q}_i$ (i.e. that of $I$) and an ideal $Q^{\prime}_i$ such that $\overline{Q}_1=Q_i\cap Q^{\prime}_i$ by a maximal independent set of $\overline{Q}_i$ (see Procedure 3.3 in \cite{SHIMOYAMA1996247}). Thus, applying the pseudo-primary decomposition recursively for $Q^{\prime}_i$ and $I^\prime$, we obtain a primary decomposition $\mathcal{Q}$ of $I$. However, $\mathcal{Q}$ may have an unnecessary component i.e. $\mathcal{Q}$ is not necessarily minimal. In order to solve this problem, we use ``saturated separating ideal'' proposed in \cite{KAWAZOE20111158}. 

\begin{Definition}[\cite{KAWAZOE20111158}, Definition~1]
    Let $I$ and $Q$ be ideals satisfying $I \subset Q$. An ideal $J$ is called {\em a separating ideal} for $(I, Q )$ if $I=Q\cap (I+J)$ holds. If a separating ideal for $(I,Q)$ satisfies $\sqrt{I:Q}=\sqrt{I+J}$ then $J$ is called {\em a saturated separating ideal} for $(I , Q )$.
\end{Definition} 

The following proposition states that every isolated primary component of $I+J$ is a primary component of $I$ for a saturated separating ideal $J$ for $(I,Q)$. 

\begin{Proposition}[\cite{KAWAZOE20111158}, Theorem~7] \label{prop:thm7}
    Suppose that $I = Q\cap J$ and $\sqrt{J} =\sqrt{I:Q}$ for a proper ideal $J$. Let $Q_1.\ldots,Q_r$ be the set of all isolated primary components of $J$ and set $Q^\prime=Q\cap \bigcap_{i=1}^r Q_i$. If $I=Q^\prime\cap J^\prime$ and $\sqrt{J^\prime}=\sqrt{I:Q^\prime}$ for a proper ideal $J^\prime$, then any minimal associated prime of $J^\prime$ is a non-minimal associated prime of $J$. 
\end{Proposition}

The existence of a saturated separating ideal for $(I,Q)$ is ensured by the following proposition. 

\begin{Proposition}[\cite{KAWAZOE20111158}, Theorem~4] \label{lem:cor5}
    Let $J$ be a separating ideal for $(I,Q)$. If $f\in \sqrt{I:Q}$ then there exists a positive integer $m$ satisfying $I=Q\cap (I+J+\langle f^m\rangle)$. 
\end{Proposition}

Here, Algorithm~\ref{alg:sy} is a derivative version of SY-algorithm, which outputs a minimal primary decomposition of a given ideal. 

\begin{algorithm}[H]      
\caption{\textsc{Shimoyama-Yokoyama (SY)}}                                 
\begin{algorithmic}[1]       \label{alg:sy}
\REQUIRE $I$: an ideal of $K[X]$
\ENSURE a minimal primary decomposition of $I$
\STATE $\mathcal{Q}=\{\}$
\STATE $P_1,\ldots,P_r\gets $ isolated prime divisors of $I$
\STATE $\{S_1,\ldots,S_r\}\gets $ a system of separators of $I$
\FOR{$i=1$ to $r$}
    \STATE $s_i\gets \prod_{s\in S_i}s$
    \STATE $\overline{Q}_i\gets I:s_i^\infty$
\ENDFOR
\FOR{$i=1$ to $r$}
        \STATE $Q\gets $ the isolated primary component of $\overline{Q}_i$
        \STATE $J_1\gets $ a saturated separating ideal for $(\overline{Q_i},Q)$
        \IF{$I+J_1\neq K[X]$}
         \STATE $\mathcal{Q}\gets \mathcal{Q}\cup \textsc{SY}(I+J_1)$
        \ENDIF
\ENDFOR
\STATE $J_2\gets $ a saturated separating ideal for $(I,\bigcap_{i=1}^r \overline{Q}_i)$
\IF{$I+J_2\neq K[X]$}
         \STATE $\mathcal{Q}\gets \mathcal{Q}\cup \textsc{SY}(I+J_2)$
\ENDIF
\RETURN $\mathcal{Q}$
\end{algorithmic}
\end{algorithm}

\subsection{Symmetric Shimoyama-Yokoyama Algorithm}

Here, we introduce an effective SY-algorithm specialized for symmetric ideals. For such specialization and {\em symmetric pseudo-primary decompositions}, we need to prove the computability of the followings:
\begin{enumerate}
    \item a symmetric system of separators (see Proposition~\ref{prop:symsep}),
    \item a symmetric saturated separating ideal (see Theorem~\ref{thm:symsatsep}).
\end{enumerate}
First, we show that there exists a symmetric system of separators as follows. 

\begin{Proposition}\label{prop:symsep}
    Let $I$ be a $G$-invariant ideal, which is not a pseudo-primary ideal. Let $P_1,\ldots,P_r$ be isolated prime divisors. There exists a system of separators $\{S_1,\ldots,S_r\}$ such that $G$ acts on $\{S_1,\ldots,S_r\}$. 
\end{Proposition}

\begin{proof}
    Let $L=\{1,\ldots,r\}$ and $S_1$ be an arbitrary separator with respect to $P_1$. Then, $\sigma(S_1)$ is a separator with respect to $\sigma(P_1)$ for $\sigma\in G$. Indeed, $\sigma (S_1)\cap \sigma(P_1)=\sigma (S_1\cap P_1)=\emptyset$ and  $\sigma(S_1)\cap \sigma (P_j)=\sigma (S_1\cap P_j)\neq \emptyset$ for $j\neq 1$. Set $L_1=L\setminus \{i\mid P_i=\sigma (P_1) \text{ for some }\sigma\in G\}$. If $L_1\neq \emptyset$, then we pick $i_1\in L_1$ and let $S_{i_1}$ be an arbitrary separator with respect to $P_{i_1}$. Then, $\sigma(S_{i_1})$ is a separator with respect to $\sigma(P_{i_1})$ for $\sigma\in G$. Inductively, for a positive integer $j\ge 2$, we define $L_j=L_{j-1}\setminus \{i\mid P_i=\sigma (P_{i_{j-1}}) \text{ for some }\sigma\in G\}$ and pick $i_j\in L_j$ and a separator $S_{i_j}$ with respect to $P_{i_j}$ until $L_1\cup \cdots \cup L_N=L$ at some positive integer $N$. It follows that $\{\sigma(P_{i_k})\mid \sigma\in G,k=1,\ldots,N\}$ is a system of separators of $I$ and $G$ acts on it.  
\end{proof}

We call $\{S_1,\ldots,S_r\}$ in Proposition~\ref{prop:symsep} {\em a $G$-invariant system of separators} of $I$. For a symmetric ideal, its pseudo-primary component can be divided into two types as follows. 

\begin{Theorem} \label{thm:pseudosym}
    Let $P_1,\ldots,P_r$ be isolated prime divisor of a $G$-invariant ideal $I$ and $\{S_1,\ldots,S_r\}$ a $G$-invariant system of separators of $I$. Then, for each pseudo-primary component $\overline{Q}_i$ with respect to $S_i$, it satisfies either one of the following conditions. 
    \begin{enumerate}
        \item $\overline{Q}_i$ is $G$-invariant
        \item $\overline{Q}_i$ does not have any $G$-invariant primary components. 
    \end{enumerate}
\end{Theorem}

\begin{proof}
    Fix $i$. For $\sigma\in G$, $\sigma(\overline{Q}_i)=\sigma(I:s_i^\infty)=I:\sigma(s_i)^\infty$. Since $\{S_1,\ldots,S_r\}$ is a $G$-invariant system of separators of $I$, $\sigma(\overline{Q}_i)$ is also a pseudo-primary component of $I$. If $P_i$ is $G$-invariant, then $\overline{Q}_i$ is also $G$-invariant. Otherwise, there exists $j\neq i$ such that $\sigma(\overline{Q}_i)=\overline{Q}_j$. In this case, for any primary component $Q$ of $\overline{Q}_i$, $\sigma(Q)$ is a primary component of $\overline{Q}_j$. Since $\overline{Q}_i$ and $\overline{Q}_j$ do not have any common primary components, $Q\neq \sigma(Q)$. Therefore, $\overline{Q}_i$ does not have any $G$-invariant primary components. 
\end{proof}

Next, we can take a symmetric saturated separating ideal for $(I,Q)$ as follows. 

\begin{Theorem} \label{thm:symsatsep}
    Let $I$ and $Q$ be $G$-invariant ideals with $I\subset Q$. Let $J$ be a $G$-invariant ideal and a separating ideal for $(I,Q)$. There exist a non-negative integer $l$ and  $f_1,\ldots,f_l\in K[X]$ such that 
    \begin{enumerate}
        \item $I=Q\cap (I+J+\langle f_1,\ldots,f_l\rangle)$,
        \item $\langle f_1,\ldots,f_l\rangle$ is a $G$-invariant ideal,
        \item $J+\langle f_1,\ldots,f_l\rangle$ is a saturated separating ideal for $(I,Q)$. 
    \end{enumerate}
\end{Theorem}

\begin{proof}
    If $J$ is a saturated separating ideal for $(I,Q)$, then $l=0$ satisfies $(1)-(3)$. Suppose $J$ is not a saturated separating ideal for $(I,Q)$. Then we can take $g_1\in (I:Q)\setminus \sqrt{I+J}$. By Proposition~\ref{lem:cor5}, there exists a positive integer $m_1$ such that $I=Q\cap (I+J+\langle g_1^{m_1}\rangle)$. Since both $I:Q$ and $\sqrt{I+J}$ are $G$-invariant, $\sigma(g_1)\in (I:Q)\setminus \sqrt{I+J}$ for any $\sigma\in G$. Thus, there exists a positive integer $m_{\sigma_1}$ for some $\sigma_1\in G\setminus \{(1)\}$ such that $I=Q\cap (I+J+\langle g_1^{m_1},\sigma (g_1^{m_{\sigma_1}})\rangle)$. Repeatedly, there exist positive integers $m_{\sigma}$ $(\sigma\in G)$ such that $I=Q\cap (I+J+\langle \sigma(g_1^{m_{\sigma}})\mid \sigma \in G\rangle)$ with $m_{(1)}=m_1$. Letting $M_1=\max\{m_{\sigma}\mid \sigma\in G\}$, we obtain $I=Q\cap (I+J+\langle \sigma(g_1)^{M_1}\mid \sigma \in G\rangle)$. Here, $F_1=\{\sigma (g_1)^{M_1}\mid \sigma\in G\}$ satisfies $(1)$ and $(2)$. If $F_1$ does not satisfy $(3)$, then we can take $g_2\in (I:Q)\setminus \sqrt{I+J}$ and set $M_2$ repeatedly. As $K[X]$ is Noetherian, there exists a positive integer $N$ such that $J+\langle \sigma(g_1)^{M_1}\mid \sigma\in G \rangle+\cdots +\langle \sigma(g_N)^{M_N}\mid \sigma\in G \rangle$ is a saturated separating ideal for $(I,Q)$. 
\end{proof}

For an ideal $I$, we say that $I$ is {\em completely anti-symmetric} if $I$ does not have any symmetric primary components. In the symmetric pseudo-primary decomposition $I=\overline{Q}_1\cap \cdots \cap \overline{Q}_r \cap I^\prime$ of $I$, there are $3$-type of components by Theorem~\ref{thm:pseudosym} and Theorem~\ref{thm:symsatsep}:

\begin{enumerate}
    \item symmetric pseudo-primary component $\overline{Q}$; in this case we apply our symmetric decomposition to $\overline{Q}$ recursively.
    \item completely anti-symmetric pseudo-primary component $\overline{Q}$; in this case we apply an ordinary primary decomposition algorithm to $\overline{Q}$
    \item symmetric remaining component $I^\prime$; in this case we apply our symmetric decomposition to $I^\prime$ recursively.
\end{enumerate}

Finally, we obtain the symmetric Shimoyama-Yokoyama Algorithm as Algorithm~\ref{alg:symsy}. The following proposition ensures its termination and correctness. 

\begin{Proposition}
    Algorithm~\ref{alg:symsy} terminates in finitely many steps and outputs a minimal primary decomposition of $I$. 
\end{Proposition}
\begin{proof}
    The termination is proved in a similar way to that of Shimoyama-Yokoyama algorithm (Theorem~3.2 in \cite{SHIMOYAMA1996247}). Thus, it is enough to prove the correctness of Algorithm~\ref{alg:symsy}. Let $\mathcal{Q}$ be the output for $I$ and $G$. The element $Q$ of $\mathcal{Q}$ is in one of the following cases:
    \begin{enumerate}
        \item $Q$ is the isolated primary component of $Q_i$ (line 11)
        \item $Q$ is a primary component of $I+J_1$ (line 15)
        \item $Q$ is equal to $\sigma (Q^\prime)$ for a primary component $Q^\prime$ of $\overline{Q}_i$ and $\sigma\in G$ (line 19)
        \item $Q$ is a primary component of $I+J_2$ (line 24)
    \end{enumerate}

    In the case $(1)$, since $\overline{Q_i}$ is a pseudo primary component of $I$, $Q$ is an isolated primary component of $I$. In the case $(2)$, as $J_1$ is a $G$-invariant saturated separating ideal for $(\overline{Q_i},Q)$, $Q$ is a primary component of both $I$ and $I+J_1$. In the case $(3)$, since $I$ is $G$-invariant, $\sigma (Q^\prime)$ is a primary component of $I$. In the case $(4)$, as $J_2$ is a $G$-invariant saturated separating ideal for $(I,\bigcap_{i=1}^r \overline{Q}_i)$, $Q$ is a primary component of both $I$ and $I+J_2$. Therefore, $Q$ is a primary component in any cases and thus $\mathcal{Q}$ is a  primary decomposition of $I$. The minimality of $\mathcal{Q}$ follows from Theorem~\ref{thm:symsatsep}. 
\end{proof}

In Algorithm~\ref{alg:symsy}, \textsc{SY} is an ordinal primary decomposition algorithm (Algorithm~\ref{alg:sy}) based on Shimoyama-Yokoyama Algorithm. 

\begin{algorithm}[H]  
\caption{\textsc{SymmetricSY}}                                 
\begin{algorithmic}[1]       \label{alg:symsy}
\REQUIRE $I$: a $G$-invariant ideal of $K[X]$. $G$: a subgroup of $\mathfrak{S}_n$
\ENSURE a minimal primary decomposition of $I$
\STATE $\mathcal{Q}=\{\}$
\STATE $P_1,\ldots,P_r\gets $ isolated prime divisors of $I$
\STATE $\{S_1,\ldots,S_r\}\gets $ a $G$-invariant system of separators of $I$
\FOR{$i=1$ to $r$}
    \STATE $s_i\gets \prod_{s\in S_i}s$
    \STATE $\overline{Q}_i\gets I:s_i^\infty$
\ENDFOR
\STATE $C_1,\ldots,C_l\gets $ the orbit decomposition of $\{\overline{Q}_1,\ldots,\overline{Q}_r\}$ with respect to $G$
\FOR{$i=1$ to $l$}
    \IF{$|C_i|=1$}
        \STATE $Q\gets $ the isolated primary component of $\overline{Q_i}$ for $C_i=\{\overline{Q_i}\}$
        \STATE $\mathcal{Q}\gets \mathcal{Q}\cup \{Q\}$
        \STATE $J_1\gets $ a $G$-invariant saturated separating ideal for $(\overline{Q_i},Q)$
        \IF{$I+J_1\neq K[X]$}
         \STATE $\mathcal{Q}\gets \mathcal{Q}\cup \textsc{SymmetricSY}(I+J_1)$
         \ENDIF
    \ELSE
        \STATE Pick $\overline{Q}_i\in C_i$ and $\mathcal{Q}_i\gets \textsc{SY}(\overline{Q}_i)$
        \STATE $\mathcal{Q}\gets \mathcal{Q} \cup \bigcup_{\sigma\in G} \sigma(\mathcal{Q}_i)$
    \ENDIF
\ENDFOR
\STATE $J_2\gets $ a $G$-invariant saturated separating ideal for $(I,\bigcap_{i=1}^r \overline{Q}_i)$
\IF{$I+J_2\neq K[X]$}
         \STATE $\mathcal{Q}\gets \mathcal{Q}\cup \textsc{SymmetricSY}(I+J_2)$
\ENDIF
\RETURN $\mathcal{Q}$
\end{algorithmic}
\end{algorithm}

\section{Experiment} 
\label{sec:Experiment}

In this section, we examine the effectiveness of our algorithm in a naive computational experiment. We implement our algorithm in the computer algebra system Risa/Asir \cite{risaasir}. Here, the author implemented {\tt SY} and {\tt symSY} based on Algorithm~\ref{alg:sy} and Algorithm~\ref{alg:symsy} respectively. Both algorithms use the Risa/Asir package \texttt{noro\_pd} to compute ideal quotients, saturations, and prime decompositions of radical ideals. In order to measure the effect of using symmetry, each algorithm is implemented as simple as possible. Therefore, {\tt SY} and {\tt symSY} may be considered slower than the Shimoyama-Yokoyama algorithm already implemented in Risa/Asir. Timings in seconds are measured on a PC with AMD Ryzen Threadripper PRO 5965WX 24-Cores and 128GB memory. 

In Table 1, there are timings of \texttt{SY} and \texttt{symSY} in the ideals $I_1,\ldots,I_{10}$:
\begin{align*}
    I_1&=\langle (x_1+x_2)^3-1,x_1x_2(x_1+x_2)\rangle\subset \mathbb{Q}[x_1,x_2],\\
    I_2&=\langle \sigma(x_1^2x_2+x_1x_3)\mid \sigma\in \mathfrak{S}_3\rangle\subset \mathbb{Q}[x_1,x_2,x_3],\\
    I_3&=\langle \sigma(x_1^4x_2+x_1x_3)\mid \sigma\in \langle\langle (1\, 2\, 3)\rangle\rangle\rangle\subset \mathbb{Q}[x_1,x_2,x_3]\\ 
    &\text{ where $\langle\langle (1\, 2\, 3)\rangle\rangle=\{(1),(1\, 2\, 3),(1\, 3\, 2)\}$},\\
    I_4&=cyclic(4)=\langle c_1c_2c_3c_4-1,c_1c_2c_3+c_1c_2c_4+c_1c_3c_4+c_2c_3c_4,\\
    &c_1c_2+c_1c_3+c_1c_4+c_2c_3+c_2c_4+c_3c_4,c_1+c_2+c_3+c_4\rangle\subset \mathbb{Q}[c_1,c_2,c_3,c_4], \\
     I_5&=\langle (c_1c_2c_3c_4-1)^2,(c_1c_2c_3+c_1c_2c_4+c_1c_3c_4+c_2c_3c_4)^2,\\
    &(c_1c_2+c_1c_3+c_1c_4+c_2c_3+c_2c_4+c_3c_4)^2,(c_1+c_2+c_3+c_4)^2\rangle \subset \mathbb{Q}[c_1,c_2,c_3,c_4],\\
    I_6&=\bigcap_{\sigma\in \mathfrak{S}_4} \sigma(\langle x_1^3-1,x_2^2\rangle)\subset \mathbb{Q}[x_1,x_2,x_3,x_4],\\
    I_7&=\bigcap_{\sigma\in \mathfrak{S}_4} \sigma(\langle x_1x_2,x_3^2-x_4\rangle) \subset \mathbb{Q}[x_1,x_2,x_3,x_4], \\
     I_8&=\bigcap_{\sigma\in \mathfrak{S}_5} \sigma(\langle x_1x_2x_3,x_4^2+x_5^2\rangle)\subset \mathbb{Q}[x_1,x_2,x_3,x_4,x_5],\\
    I_9&=\bigcap_{\sigma\in \mathfrak{S}_5} \sigma(\langle x_1^2-1,x_2^3,x_3^4\rangle)\subset \mathbb{Q}[x_1,x_2,x_3,x_4,x_5],\\
    I_{10}&=\bigcap_{\sigma\in \mathfrak{S}_6} \sigma(\langle x_1^2-1,x_2^3,x_3^4\rangle)\subset \mathbb{Q}[x_1,x_2,x_3,x_4,x_5,x_6].
\end{align*}
An invariant group to each row ideal $I_i$ is in the second column. The cardinality of a minimal primary decomposition of each ideal is in the third column. Also, the number of the orbit decomposition $\{C_1,\ldots,C_l\}$ of each ideal is in the fourth column. We see that the symmetric SY-algorithm is effective for each $I_i$ compared to the ordinal SY-algorithm. While the computation time of \texttt{SY} increases as the number of primary components ($\# \mathcal{Q}$) increases, that of {\tt symSY} increases slowly since it essentially requires only $l$-primary components, where $l$ is the number of orbits ($\# Orbit$). For example, $I_9$ has $60$-primary components and {\tt SY} takes $266$ seconds for the computation. On the other hand, {\tt symSY} computes it much faster (in 7.62 seconds) since it computes $60$ primary components from $2$ primary components which are in each orbit by the group action of $\mathfrak{S}_5$. These results show that {\tt symSY} can be effective when the input ideal has many primary components and a highly symmetric structure.

	\begin{table}
 \label{table1}
		\begin{center}
\begin{tabular*}{10cm}{|c@{\extracolsep{\fill}}ccccc|D{.}{.}{-1}||D{.}{.}{-1}|}\hline
&ideal & $G$ & $\# \mathcal{Q}$ & $\# Orbit$ && \multicolumn{1}{c||}{{\tt SY}} & \multicolumn{1}{c|}{{\tt symSY}} \\\hline
&$I_1$ & $\mathfrak{S}_2$  & $4$ & $2$&& 0.01&0.01\\
&$I_2$ & $\mathfrak{S}_3$  &$7$ & $3$&& 0.03&0.03\\ 
&$I_3$ & $\langle\langle (1\, 2\, 3)\rangle\rangle$  &$15$ & $7$&& 0.10&0.09\\ 
&$I_4$ & $\mathfrak{S}_4$  &$8$ & $3$&& 0.07&0.05\\
&$I_5$ & $\mathfrak{S}_4$  &$8$ & $3$&& 0.65&0.40\\
&$I_6$ & $\mathfrak{S}_4$ &$24$&$2$&&2.68&0.48 \\
&$I_7$ & $\mathfrak{S}_4$ &$24$&$1$&&796&42.2 \\
&$I_8$ & $\mathfrak{S}_5$ &$30$&$1$&&91.7&3.58 \\
&$I_9$ & $\mathfrak{S}_5$ &$60$&$2$&&266&7.62 \\
&$I_{10}$ & $\mathfrak{S}_6$ &$120$&$2$&& \multicolumn{1}{c||}{$>$ 5 days}&6377 \\
\hline
\end{tabular*}			
		\end{center}
  \caption{Timings of \texttt{SY} and \texttt{symSY}}
	\end{table}	

\section{Conclusions and Future Works}
\label{sec:ConclusionsandFutureWorks}
Symmetric ideals appear in various areas of mathematics. In this paper, we prove good properties of symmetric ideals in order to provide an effective algorithm for a primary decomposition of such an ideal. In the proposed algorithm, one can compute full primary components from partial ones by the group action on the ideal. For practical computations, we devise Shimoyama-Yokoyama Algorithm specialized to symmetric ideals. In a computational experiment, we examine that the specialized algorithm is faster than the ordinal one in several cases. 

For future works, we plan to improve the algorithm and apply it to examples that appear in several areas e.g. statistics. For improvements,  we are thinking of devising a specialized algorithm of Kawazoe-Noro algorithm \cite{KAWAZOE20111158} to compute efficiently for more types of symmetric ideals. Also, we will consider generalized group actions, for example, the general linear group $GL(n,K)$ on $K[X]$ by $A( f(x_1,\ldots,x_n))=f(A(x_1,\ldots,x_n))$ for a matrix $A\in GL(n,K)$ and a polynomial $f\in K[X]$. 

\section*{Acknowledgements}
The author would like to thank Kazuhiro Yokoyama and Yuta Kambe for the fruitful discussion. He is grateful to the anonymous referee for helpful comments and suggestions. This work was supported by JSPS KAKENHI Grant Number JP22K13901.

\end{document}